\documentclass{amsart}
\title[Transverse K\"ahler structures on central foliations]
{Transverse K\"ahler structures on central foliations of  complex manifolds}

\author{Hiroaki Ishida}
\address{Department of Mathematics and Computer Science, Graduate School of Science and Engineering, Kagoshima University}
\email{ishida@sci.kagoshima-u.ac.jp}
\author{Hisashi Kasuya}
\address{Department of Mathematics, Graduate School of Science, Osaka University, Osaka, Japan.}
\email{kasuya@math.sci.osaka-u.ac.jp}

\date{\today}
\thanks{The first author was supported by JSPS Grant-in-Aid for Young Scientists (B) 16K17596. }

\keywords{transverse K\"{a}hler structure, central foliation, basic cohomology, basic Dolbeault cohomology, mixed Hodge structure}
\subjclass[2010]{Primary 32Q55, Secondary 37F75, 55P62, 57N65, 58A14}

\usepackage{amssymb}
\usepackage{amsmath}
\usepackage{amscd}
\usepackage{amstext}
\usepackage{amsfonts}
\usepackage[all]{xy}
\usepackage{color}

\newif\ifdebug

\debugfalse

\newcommand{\ishida}[1]{\ifdebug{{\ \scriptsize{\color{red}{石田:\ #1}}\ }}\fi}

\newcommand{\kasuya}[1]{\ifdebug{{\ \scriptsize{\color{red}{糟谷:\ #1}}\ }}\fi}

\newcommand{\printname}[1]{\smash{\makebox[0pt]{\hspace{-1.0in}\raisebox{8pt}{\tiny #1}}}}

\newcommand{\Label}[1]{\ifdebug{\label{#1}\printname{#1}}\else{\label{#1}}\fi}

\newtheorem{prop}{Proposition}[section]
\newtheorem*{prop*}{Proposition \referenza}
\newtheorem{thm}[prop]{Theorem}
\newtheorem*{thm*}{Theorem \referenza}
\newtheorem{cor}[prop]{Corollary}
\newtheorem*{cor*}{Corollary \referenza}
\newtheorem{lemma}[prop]{Lemma}

\theoremstyle{definition}
\newtheorem{defi}[prop]{Definition}
\newtheorem{rem}[prop]{Remark}
\newtheorem{ex}[prop]{Example}
\newtheorem{pro}[prop]{Problem}
\newcommand{\C}{\mathbb{C}}
\newcommand{\R}{\mathbb{R}}

\newcommand{\Q}{\mathbb{Q}}

\newcommand{\g}{\frak{g}}

\newcommand{\h}{\frak{h}}

\begin{document} 

\maketitle

\begin{abstract}
For a compact complex manifold, we introduce holomorphic  foliations associated with certain abelian subgroups of the automorphism group.
If there exists a transverse K\"ahler structure on such a foliation, then we obtain a nice differential graded algebra which is quasi-isomorphic to the de Rham complex and a nice differential bigraded algebra which is quasi-isomorphic to the Dolbeault complex like the formality of compact K\"ahler manifolds.
Moreover,  under certain additional condition, we can develop Morgan's  theory of mixed Hodge structures as similar to the study on smooth algebraic varieties.

\end{abstract}
\section{Introduction}
	When a connected Lie group $H$ acts on a smooth manifold $M$ local freely, \ishida{thenを取りました} we have a smooth foliation $\mathcal{F}$ whose leaves are $H$-orbits. In addition, if $M$ is a complex manifold, $H$ is a complex Lie group and the $H$-action is holomorphic, then the foliation $\mathcal{F}$ is holomorphic. 
We are interested in transverse complex geometry on a foliated manifold $(M,\mathcal F)$.
Denote by $\Omega^{\ast}(M)$ the space of the differential forms on $M$.
We say that $\omega\in \Omega^{\ast}(M) $ is {\em basic} if $i_{X_{v}}\omega=0$ and $L_{X_{v}}\omega=0$ for any $v\in \frak h$, where \ishida{追加} $X_v$ denotes the fundamental vector field generated by $v \in \frak h$, and \ishida{追加ここまで} $i_{X_{v}}$ and  $L_{X_{v}}$ are the intrerior \ishida{old:inner} product and the Lie derivation \ishida{with $X_v$を追加} with $X_v$, respectively.
For a holomorphic foliation $\mathcal  F$  on a complex manifold $M$ with the complex structure $J$,  a transverse K\"ahler structure on $\mathcal F$ is a 
closed real basic $(1,1)$-form $\omega$
such that $\omega_p (v,Jv) \geq 0$ for any $v \in T_pM$ and $p \in M$, and the equality holds if and only if $v$ sits in the subspace $T_p\mathcal{F}$ that consists of all vectors tangent to the leaf through $p$. 

In this paper we introduce an intrinsically defined holomorphic foliation for arbitrary compact complex manifold, that we call \emph{canonical foliation}. Let $M$ be a compact complex manifold.
Let $G_{M}$ be the identity component of the group of all biholomorphisms on $M$.  Let $T$ be a maximal compact torus of $G_{M}$ and $\mathfrak{t}$ the Lie algebra of $T$. Let $J$ be the complex structure on the Lie algebra of $G_M$. Put 
\begin{equation*}
	\h_{M} := \mathfrak{t} \cap J \mathfrak{t}
\end{equation*}
and denote by $H_{M}$ the corresponding Lie subgroup of $G_{M}$. Then $H_{M}$ acts on $M$ local freely (see \cite{I}). 
Moreover $H_{M}$ is a central subgroup  in $G_M$ and $H_{M}$  does not depend on the choice of $T$ (see Lemma \ref{lemmaonH}). 
By the local freeness, for any connected subgroup $H\subset H_{M}$, we have the holomorphic foliation ${\mathcal  F}_{H}$.
We call ${\mathcal  F}_{H}$ a {\em central foliation} associated with $H$ and ${\mathcal F}_{H_M}$ the \emph{canonical foliation}. 
If $H$ is a compact complex torus, 
then the central foliation ${\mathcal  F}_{H}$  associated with $H$ gives a holomorphic principal Seifert bundle structure on a complex manifold $M$ over the complex orbifold $M/H$ (see \cite{Or}).
Moreover, if the action of $H$ is free, then such Seifert bundle is  a  holomorphic principal torus bundle over a complex manifold.
Conversely,   holomorphic principal torus bundle structure gives a holomorphic free complex torus action.
Thus, a central foliation is a generalization of a holomorphic principal torus bundle.

The purpose of this paper is to study (non-K\"ahler) complex manifolds admitting a transverse K\"ahler  structure on a central foliation ${\mathcal  F}_{H}$.  Especially, we study the de Rham and Dolbeault complexes of such complex manifolds. 
Typical  examples are holomorphic principal torus bundles over compact K\"ahler manifolds. 
A Calabi-Eckmann manifold is a holomorphic principal torus bundle over $\C P^m \times \C P^n$, and its underlying smooth manifold is diffeomorphic to $S^{2m+1} \times S^{2n+1}$. 
Extending  Calabi-Eckmann's construction, Meersseman  constructed a large class of non-K\"{a}hler compact complex manifolds.
Such complex manifolds are called \emph{LVM manifolds} (see \cite{M}, \cite{MV}).
Every LVM manifold admits a  transverse K\"ahler  structure (see \cite{M}). 
Among LVM manifolds with the canonical foliations, some are principal Seifert bundles. At that time the leaf space $M/\mathcal{F}_{H_M}$ is a projective toric variety (\cite{MV}). However there are many LVM manifolds with the canonical foliations which are not principal Seifert bundles. 

 \kasuya{追加}We notice that our object appears in a \ishida{specialをcertainに変更}certain  non-K\"aher Hermitian manifold.
 A Vaisman manifold is a non-K\"aher locally conformal K\"ahler  manifold with the  non-zero parallel Lee form.
There are many important non-K\"ahler manifolds which are Vaisman (e.g. Hopf manifolds, Kodaira-Thurston manifolds).
On any Vaisman manifold, there exists a   complex $1$-dimensional central foliation with a transverse K\"ahler  structure
which is canonically determined by its Vaisman structure.

For a complex manifold $M$ with a holomorphic foliation $\mathcal F$, 
we consider the basic de Rham complex $\Omega^{*}_{B}(M)$, basic Dolbeault complex $\Omega_{B}^{*,*}(M)$, 
 basic de Rham cohomology $H^{\ast}_{B}(M)$ and basic Dolbeault cohomology $H^{\ast,\ast}_{B}(M)$ for  $\mathcal F$.
If there exists a transverse K\"{a}hler form with respect to $\mathcal F$ and $\mathcal F$ is homologically oriented, then there is the Hodge decomposition
\[H^{r}_{B}(M,\C)=\bigoplus_{p+q=r}H^{p,q}_{B}(M),
\] 
\[\overline{H^{p,q}_{B}(M)}=H^{q,p}_{B}(M)
\]
(see \cite{EKA}).

\begin{defi}
\begin{itemize}
\item
For a manifold $M$, a (de Rham) model of $M$ is a differential graded algebra (shortly DGA) $A^{\ast}$
such that  $A^{\ast}$ is quasi-isomorphic to the de Rham complex $\Omega^{\ast}(M)$ i.e.\ there exists a sequence of DGA homomorphisms
\[A^{*}\leftarrow C^{*}_{1}\rightarrow C^{*}_{2}\leftarrow\cdot \cdot \cdot \leftarrow C^{*}_{n}\rightarrow \Omega^{*}(M)
\]
such that all the morphisms  are  quasi-isomorphisms (i.e.\ inducing cohomology isomorphisms).
\item 
For a complex manifold $M$, a Dolbeault-model of $M$ is a differential bi-graded algebra (shortly DBA) $B^{\ast,\ast}$ such that  $B^{\ast,\ast}$ is quasi-isomorphic to the Dolbeault complex $\Omega^{\ast,\ast}(M)$
i.e.\ there exists a sequence of DBA homomorphisms
\[B^{*,*}\leftarrow C^{*,*}_{1}\rightarrow C^{*,*}_{2}\leftarrow\cdot \cdot \cdot \leftarrow C^{*,*}_{n}\rightarrow \Omega^{*,*}(M)
\]
such that all the morphisms  are quasi-isomorphisms.

\end{itemize}
\end{defi}

On a compact K\"ahler manifold $M$, the de Rham cohomology $H^{*}(M)$ with the trivial differential is a  model of $M$ (Formality \cite{DGMS}) and the Dolbeault cohomology $H^{\ast,\ast}(M)$ with the trivial differential  is  a Dolbeault-model of $M$ (Dolbeault-Formality \cite{NT}).
In this paper we prove:

\begin{thm}[See also Theorem \ref{theorem-model}]
Let  $M$ be a compact complex manifold.
We assume that $M$ admits a   transverse K\"ahler structure on a complex $k$-dimensional central foliation ${\mathcal F}_{H}$.
Then there exists a model $A^*$ of $M$ with a differential $d$ and a Dolbeault-model $B^{*,*}$ of $M$ with a differential $\bar \partial$ satisfying the followings:
\begin{enumerate}
	\item Let $W$ be a real $2k$-dimensional vector space with a direct sum decomposition $W\otimes \C=W^{1,0}\oplus W^{0,1}$ satisfying $\overline{W^{1,0}}=W^{0,1}$.
As graded algebras, $A^* =H^*_B(M)\otimes \bigwedge W$. As bi-graded algebras, $B^{*,*} = H^{*,*}_B(M) \otimes \bigwedge (W^{1,0} \oplus W^{0,1})$. Here, the degree of an element in $W$ is $1$ and bi-degree of an element in $W^{1,0}$ (respectively, $W^{0,1})$ is $(1,0)$ (respectively, $(0,1)$). 
	\item The differentials $d$ and $\bar\partial$ are trivial on $H^*_B(M)$ and $H^{*,*}_B(M)$ respectively. $dW \subset H^2_B(M)$, $\bar\partial W^{1,0}\subset H^{1,1}_B(M)$ and $\bar\partial W^{0,1}\subset H^{0,2}_B(M)$. 
\end{enumerate}
\end{thm}
In \cite{Ta}, Tanr\'e constructed a Dolbeault-model for a holomorphic principal torus bundle over a compact K\"ahler manifolds. The theorem above slightly generalizes the result of Tanr\'e. 

\ishida{追加} More precisely, a vector $W$ as in the theorem is a $2k$-dimensional subspace of $\Omega^1(M)^H$ such that the bilinear map $\mathfrak{h} \times W \ni (v,w) \mapsto i_{X_v}w \in \R$ is non-degenerate, where $\mathfrak{h}$ is the Lie algebra of $H$ and $\Omega^1(M)^H$ is the $H$-invariant subspace of $\Omega^1(M)$. We can choose $W$ to be closed under the complex structure of $\Omega^1(M)$. Then $W^{1,0}$ and $W^{0,1}$ are defined as $(1,0)$-part and $(0,1)$-part of $W \otimes \C$, respectively. 
The differential $W \to H^2_B(M)$ is given by $w \mapsto [dw]_B$, where $[dw]_B$ denotes the basic cohomology class represented by $dw \in \Omega^2_B(M)$. Similarly, the differentials $W^{1,0} \to H^{1,1}_B(M)$ and $W^{0,1} \to H^{0,2}_B(M)$ are given by $w \to [\bar\partial w]_B$, where where $[\bar\partial w]_B$ denotes the basic Dolbeault cohomology class represented by $\bar\partial w \in \Omega^2_B(M) \otimes \C$. 

\ishida{old: More precisely, a vector space $W$ as in the theorem is a $k$-dimensional $\C$-subspace $W\subset \Omega^{1}(M)^{H}$ 
such that the bilinear map $\h\times W\ni (v,w)\mapsto i_{X_v}w\in \R$ is non-degenerate, where $\frak h$ is the Lie algebra of $H$ and  $\Omega^{*}(M)^{H}$ is the space of the $H$-invariant differential forms.
The differential $W\to H^{2}_{B}(M)$ (resp.\ $ W^{1,0}\to  H^{1,1}_{B}(M)$ and $ W^{0,1}\to H^{0,2}_{B}(M)$) for ${A}^{*}=H^{*}_{B}(M)\otimes \bigwedge W$ (resp.\ ${B}^{*,*}=H^{*,*}_{B}(M)\otimes \bigwedge (W^{1,0}\oplus W^{0,1})$) is given by
$W\ni w\mapsto [dw]_{B}\in H^{2}_{B}(M)$ (resp.\ $W^{1,0}\ni w\mapsto [\bar\partial w]_{B}\in H^{1,1}_{B}(M)$ and $W^{0,1}\ni w\mapsto [\bar\partial w]_{B}\in H^{0,2}_{B}(M)$).}

\kasuya{追加}
By these results, we can construct explicit  de Rham and Dolbeault models of \ishida{vaismanをVaismanにした}Vaisman manifolds \ishida{括弧をつけて前の文にひっつけました}(see Section \ref{vais}).
Recently, similar de Rham models are also constructed  in \cite{CNMY}.

\begin{defi}\label{speci}
A central foliation ${\mathcal  F}_{H}$
 is {\em fundamental} if  for any $w\in W$, $[dw]_{B}\in H^{2}_{B}(M)$  is represented  by a closed basic $(1,1)$-form. 
\end{defi}

We prove:
\begin{thm}[See also Theorem \ref{theorem-MHS}]

Let  $M$ be a compact complex manifold.
We assume that $M$ admits a    transverse K\"ahler structure on a fundamental central foliation ${\mathcal  F}_{H}$.
Then the de Rham cohomology  of $M$ admits an $\R$-mixed Hodge structure so that:
\begin{enumerate}
\item $H^{1}(M,\C)=H^{1}_{1,0}\oplus H^{1}_{0,1}\oplus H_{1,1}^{1}$
\item $H^{2}(M,\C)=H^{2}_{2,0}\oplus H^{2}_{1,1}\oplus H_{0,2}^{2}\oplus H^{2}_{2,1}\oplus H^{2}_{1,2}\oplus H_{2,2}^{2}$
\end{enumerate}
 and  Sullivan's minimal model of the complex valued de Rham complex admits the Morgan's bigrading (\cite{Mor}).
\end{thm}
As a consequence of this result,
 we can say that  not every finitely generated group can be the fundamental  group of a compact complex  manifold admitting a   transverse K\"ahler structure on a fundamental  central foliation.

\section{Central foliations}
Let $M$ be a compact complex manifold.
In this section we define the \emph{canonical} foliation and \emph{central} foliations on $M$. 
Let $G_{M}$ be the identity component of the group of all biholomorphisms on $M$.
 $G_{M}$ is a complex Lie group (see \cite{BM}).  
 Denote by $\g_{M}$ the Lie algebra of $G_{M}$ and by $J$ the complex structure on $\g_M$. 
 Let $T$ be a maximal compact torus of $G_M$ and $\mathfrak{t}$ the Lie algebra of $T$. Put 
\begin{equation*}
	\h_{M}:= \mathfrak{t} \cap J \mathfrak{t}
\end{equation*}
and denote by $H_{M}$ the corresponding Lie subgroup of $G_{M}$. Then $H_{M}$ acts on $M$ local freely (see \cite[Proposition 3.3]{I}). 
\begin{lemma}\Label{lemmaonH}
	The following holds: 
\begin{enumerate}
	\item Elements in $\mathfrak{h}_{M}$ centralize $\g_{M}$. 
	\item $\mathfrak{h}_{M}$ does not depend on the choice of $T$. 
\end{enumerate}
\end{lemma}
\begin{proof}
	Since $T$ is compact, $\g_{M}$ is a unitary representation of $T$. In particular, $\g_{M}$ is a unitary representation of $H_{M}$. However, $H_{M}$ is a holomorphic subgroup of $G_{M}$ and hence $\g_{M}$ is a holomorphic representation of $H_{M}$. Therefore $\g_{M}$ is a trivial representation of $H_M$, showing Part (1). 
	
	Let $T'$ be another maximal compact torus of $G_{M}$. Then, there exists $g \in G_{M}$ such that $gTg^{-1} = T'$ (see \cite[Chapter XV, Section 3]{Hochschild} for detail). Put 
	\begin{equation*}
		\mathfrak{h}' := \mathfrak{t}' \cap J\mathfrak{t}'. 
	\end{equation*}
	Then, it follows from $gTg^{-1} = T'$ that $\operatorname{Ad}_g (\mathfrak{h}_{M}) =  \mathfrak{h}'$. On the other hand, by (1), we have \ishida{that 追加} that $\operatorname{Ad}_g $ is the identity on $\mathfrak{h}_{M}$. Therefore $\mathfrak{h}_{M}$ does not depend on the choice of $T$, proving (2).
\end{proof}
	We remark that any $\C$-subspace $\mathfrak{h}$ of $\mathfrak{h}_M$ defines a holomorphic foliation $\mathcal{F}_H$ on $M$. We call $\mathcal{F}_H$ a \emph{central foliation} on $M$ and $\mathcal{F}_{H_M}$ the \emph{canonical foliation} on $M$. 
	 It follows from Lemma \ref{lemmaonH} that the canonical foliation does not depend on the choice of $T$, that is, the canonical foliation is intrinsic \ishida{to compact complex manifolds 追加} to compact complex manifolds. 

\section{Hirsch extensions and minimal models}
In this section, DGAs are defined over ${\mathbb K}=\Q, \R$ or $\C$, if we do not specify.
Let $(A^{*},d_{A})$ be a DGA.
\ishida{old: Let $V$ be a vector space of degree $k$.} Let $k$ be an integer. 
For a linear map $\beta:V \to A^{k+1}$ with $d_{A}\circ \beta=0$,
we define a Hirsch extension $(B,d_{B})$ of $A^{*}$ in degree $k$ such that $B^{*}=A^{*}\otimes \bigwedge V$ with $\deg(v)=k$ for any $v\in V$, 
$d_{B}=d_{A}$ on $A^{*}$ and $d_{B}=\beta $ on $V$.
Defining the filtration on $B^{*}$ by $F^{p}(B^{*})=A^{*\ge p}\otimes  \bigwedge V$,
we have the spectral sequence $E^{*,*}$ with
$E^{p,q}_{2}=H^{p}(A)\otimes \bigwedge^{q}V$.
Consider the composition $q\circ \beta :V\to H^{k+1}(A^{*})$ where $q:\ker d_{A}\to H^{*}(A^{*})$ is the quotient \ishida{map 追加} map. 
The DGA structure of $B^{*}$ is determined by the map   $q\circ \beta$ (independent of a choice of $\beta$) (\cite[10.2]{GMo}).

\begin{lemma}\label{hir}
Let $(A^{*}_{1},d_{A_{1}})$ and $(A^{*}_{2},d_{A_{2}})$ be DGAs and $f:A_{1}^{*}\to A_{2}^*$ a quasi-isomorphism.
Then for a Hirsch extension $A^{*}_{1}\otimes V$ (resp.\ $A^{*}_{2}\otimes V$),   we have a Hirsch extension $A^{*}_{2}\otimes V$ (resp.\ $A^{*}_{1}\otimes V$) and quasi-isomorphism
\[A^{*}_{1}\otimes V\to A^{*}_{2}\otimes V.
\]

\end{lemma}

\begin{proof}
In case $A^{*}_{1}\otimes V$ ($\beta_{1}:V\to A^{k+1}_{1} $) is given.
Consider the Hirsch extension $A^{*}_{2}\otimes V$ given  by $\beta_{2}=f\circ \beta_{1}:V\to A^{*}_{2}$ and the homomorphism $f\otimes {\rm id}: A^{*}_{1}\otimes V\to A^{*}_{2}\otimes V$.
Then we can easily show that $f\otimes {\rm id}$ induces an isomorphism on the $E_{2}$-term of the spectral sequence.
Hence $f\otimes {\rm id}$ is a quasi-isomorphism.

In case $A^{*}_{2}\otimes V$ ($\beta_{2}:V\to A^{k+1}_{2} $) is given.
Since $f$ is a quasi-isomorphism, 
we can take a linear map $\beta_{1}:V\to A^{*}_{1}$ so that $d\circ \beta_{1}=0$ and
$q\circ f\circ\beta_{1} =q\circ \beta_{2}$.
\ishida{old: By the above argument, } By the same argument as above, the Hirsch extension of $A^{*}_{2}$ given  by $\beta_{2}:V\to A^{k+1}_{1} $ is identified  with the one given by $ f\circ\beta_{1}: V\to A^{k+1}_{1} $.
Under this identification, we have the homomorphism
$f\otimes {\rm id}: A^{*}_{1}\otimes V\to A^{*}_{2}\otimes V$
and as in the first case we can show that this homomorphism is a quasi-isomorphism.

\end{proof}

\begin{defi}
A DGA $\mathcal M^{\ast}$  is minimal if:
\begin{itemize}
\item  $\mathcal M^{0}=\mathbb K$. 
 \item
$\mathcal M^{\ast}=\bigcup\mathcal M^{\ast}_{i}$ for a sequence of sub-DGAs
\[\mathbb K=\mathcal M^{\ast}_{0}\subset \mathcal M^{\ast}_{1}\subset \dots
\]
such that  $\mathcal M_{i+1}^{\ast}$ is a Hirsch extension of  $\mathcal M_{i}^{\ast}$. 

\item $d{\mathcal  M}^{*}\subset {\mathcal  M}^{+}\cdot {\mathcal  M}^{+}$
where $M^{+}=\bigoplus_{j>0}{\mathcal M}^{j}$.
\end{itemize}

We say that a DGA $\mathcal M^{\ast}$  is $k$-minimal if $\mathcal M^{\ast}$ is minimal and $\bigoplus_{j>k}{\mathcal M}^{j} \subset  {\mathcal  M}^{+}\cdot {\mathcal  M}^{+}$. Equivalently each extension in  a sequence for $\mathcal M^{\ast}$   has degree at most $k$.
\end{defi}

\begin{defi}
Let $A^{\ast}$ be a DGA with $H^{0}(A^{\ast})=\mathbb K$.
\begin{itemize}
\item
A minimal DGA $\mathcal M^{*}$ is \ishida{old:the} a {\em minimal model}  of $A^{\ast}$ if there is a quasi-isomorphism  $\mathcal M\to A^{\ast}$.
\item A $k$-quasi-isomorphism $\mathcal{M}^* \to A^*$ is a homomorphism of DGAs that induces an  isomorphism $H^{j}(\mathcal M^{*})\cong H^{j}(A^{*})$ for $j\le k$ and an injection $H^{k+1}(\mathcal M^{*})\hookrightarrow H^{k+1}(A^{*})$. A $k$-minimal DGA $\mathcal M^{*}$ is the $k$-{\em minimal model}  of $A^{\ast}$ if there is a $k$-quasi-isomorphism
  $\mathcal M^{*}\to A^{\ast}$.

\end{itemize}

\end{defi}

\begin{thm}[\cite{Sul}]
For a DGA $A^*$ with $H^0(A^*) = \mathbb K$, \ishida{old:the} a minimal model and \ishida{old:the} a $k$-minimal model exist and each of them is unique up to DGA isomorphism. 

\end{thm}

The minimal models give the following ``de Rham homotopy theorem".
We shall state it but omit the details. 
See \cite{Sul},\cite{DGMS}, \cite{GMo}, \cite{Mor} for the details.

\begin{thm}\label{dHom}
Let $M$ be a compact smooth manifold.
Consider the de Rham complex $\Omega^{*}(M)$ as a DGA.
Then \begin{itemize}
\item The $1$-minimal model of $\Omega^{*}(M)$ is the dual to the Lie algebra of the nilpotent completion of $\pi_1(M)$. 

\item If $M$ is simply connected, then the minimal model of $\Omega^{*}(M)$ determines the real homotopy type of $M$.

\end{itemize}
\end{thm}

\section{Models for transverse K\"ahler torus actions}
In this section, we give a model and Dolbeault-model of a complex manifold equipped with a central foliation. 
\subsection{Models for compact Lie group actions}
The following result is well-known (see \cite{Fe} for example).
\begin{prop}\Label{cptderham}
Let $M$ be a compact manifold and $K$ a compact connected
Lie group.
Assume that $K$ acts on $M$.
Then the inclusion
\[\Omega^{*}(M)^{K}\subset \Omega(M)
\]
induces a cohomology isomorphism.
\end{prop}

Let $M$ be a complex manifold.
Let $(\Omega^{*,*}(M),\bar\partial)$ be the Dolbeault complex of $M$.
Suppose that a group $K$ acts on $M$ as biholomorphisms.
Then the space $\Omega^{*,*}(M)^{K}$ of $K$-invariant differential forms  is a subcomplex of $\Omega^{*,*}(M)$.

\begin{prop}\Label{prop:torusdol}
Let $M$ be a compact complex manifold and $K$ a connected compact  Lie group.
Assume that $K$ acts on $M$ as biholomorphisms and the induced action on the Dolbeault cohomology is trivial.
Then the inclusion
\[\Omega^{*,*}(M)^{K}\subset \Omega^{*,*}(M)
\]
induces an isomorphism on Dolbeault cohomology.
\end{prop}
\begin{proof}
Let $d\mu$ be the normalized Haar measure of $K$. 
Define the linear map 
\[I:\Omega^{*,*}(M)\ni \omega\mapsto \int_{g\in K}g^{*}\omega d\mu\in  \Omega^{*,*}(M)^{K}.
\]
Then $I$ commutes with Dolbeault operator $\bar \partial$, that is, $I$ induces a bi-graded module homomorphism $H^{*,*}(M) \to H(\Omega^{*,*}(M)^K)$.
Since $I$ is the identity on $\Omega^{*,*}(M)^K$, the composition 
\[
	H(\Omega^{*,*}(M)^K) \to H^{*,*}(M) \to H(\Omega^{*,*}(M)^K)
\]
of the homomorphisms induced by the inclusion and $I$ is the identity. Therefore the homomorphism $H(\Omega^{*,*}(M)^K) \to H^{*,*}(M)$ induced by the inclusion $\Omega^{*,*}(M)^{K}\subset \Omega^{*,*}(M)$ is injective.

Since  the induced action on the Dolbeault cohomology is trivial,
for a $\bar \partial$-closed form $\omega\in \Omega^{*,*}(M)$ and any $g\in G$, there exists $\theta_{g}\in  \Omega^{*,*}(M)$ such that 
\[\omega-g^*\omega=\bar \partial\theta_{g} .
\]
By using Green operator, we can take $\theta_{g} $ smoothly  on $G$.
Integrating by $d\mu$, we have
\[\omega-I\omega=\bar\partial \int_{g\in K}\theta_{g}d\mu.
\]
Hence the inclusion
\[\Omega^{*,*}(M)^{K}\subset \Omega^{*,*}(M)
\]
induces a surjection on Dolbeault cohomology. 
\end{proof}

\begin{cor}\Label{cor:apdol}
Let $M$ be a compact complex manifold and $K$ a connected compact Lie group acting on $M$ as biholomorphisms. 
Let $H$ be a dense Lie subgroup of $K$ such that $H$ is a complex Lie group and
the restricted action of $K$ to $H$ on $M$ is holomorphic. Then, the inclusion $\Omega^{*,*}(M)^K \subset \Omega (M)$ induces an isomorphism on Dolbeault cohomology. 
\end{cor}
\begin{proof} 
By Proposition \ref{prop:torusdol},
we only need to know that the representation of $K$ on $H^{*,*}(M)$ is trivial under the assumptions of this proposition.
Since $K$ acts on $M$ as biholomorphisms, the representation of $K$ on $H^{*,*}(M)$ is $\C$-linear. Since $K$ is compact, there exists a Hermitian inner product on $H^{*,*}(M)$ that is invariant under $K$. 

Consider  the restricted representation  $H\to GL(H^{*,*}(M))$.
Since $H$ is a complex Lie group and the restricted action of $K$ to $H$ on $M$ is holomorphic,  this representation is  holomorphic (\cite{Le}).
On the other hand, by the same argument as above, this representation is unitary.
Therefore the representation of $H$ on $H^{*,*}(M)$ is trivial.
 Since $H$ is dense in $K$, the representation of $K$ on $H^{*,*}(M)$ is also trivial. The proposition is proved. 
\end{proof}

\subsection{Models for torus actions}
Let $T$ be a compact torus and $H$ a connected Lie subgroup (not necessary to be closed in $T$). Let $M$ be a paracompact smooth manifold equipped with an action of $T$.
In this section, we suppose that the restricted action of $T$ to $H$ on $M$ is local free. 
Denote by $\frak t$ and $\frak h$ the Lie algebras of $T$ and $H$ respectively.

\begin{lemma}\Label{lemma:connection}
	There exists a $\h$-valued $1$-form $\omega$ on $M$ such that 
	\begin{enumerate}
		\item $i_{X_v}\omega = v$ for all $v \in \h$, 
		\item $\omega$ is $T$-invariant. 
	\end{enumerate}
\end{lemma}

\begin{proof}
	Since $T$ is compact and $M$ is paracompact, it follows from the slice theorem that there exists a locally finite open covering $\mathcal{U} = \{ U_\lambda \}_\lambda$ such that each $U_\lambda$ is $T$-equivariantly diffeomorphic to $T \times_{T_\lambda}V_\lambda$ via $\varphi _\lambda$, where $T_\lambda$ is a closed subgroup of $T$ and $V_\lambda$ is a representation space of $T_\lambda$. 
Let $\pi : T \times _{T_\lambda}V_\lambda \to T/T_\lambda$ be the map induced by the first projection $T \times V_\lambda \to T$. Since the action of $H$ on $M$ is local free, we have that $\h \cap \frak t_\lambda = 0$. Therefore there exists a $\h$-valued $1$-form $\omega_\lambda$ on $T/T_\lambda$ that satisfies the conditions (1) and (2). Since $\pi$ and $\varphi_\lambda$ are $T$-invariant, the pull-back $(\pi \circ \varphi_\lambda)^*\omega_\lambda$ that is a $\h$-valued $1$-form on $U_\lambda$ also satisfies the conditions (1) and (2). 
	
	Let $\{\rho_\lambda\}$ be a partition of unity subordinate to the open covering $\mathcal{U}$. Averaging $\rho_\lambda$ with the normalized Haar measure on $T$, we may assume that every $\rho_\lambda$ is $T$-invariant. Then the $1$-form 
	\begin{equation*}
		\omega := \sum_{\lambda} \rho_\lambda (\pi \circ \varphi_\lambda)^*\omega_\lambda
	\end{equation*}
	on $M$ satisfies the condition (1) and (2), as required. 
\end{proof}

Since the $H$-action is local free, 
the $H$-action \ishida{old:implies} induces the foliation $\mathcal F$ \ishida{old: on $M$} whose leaves are $H$-orbits of $M$. 
Denote by $T^{\prime}$ the closure of $H$.
\begin{lemma}\Label{lemmaGeqH}
\ishida{別行立てでしたが, これをやめました} 
$\Omega^{*}(M)^{T^{\prime}}=\Omega^{*}(M)^{H}$. 
\end{lemma}
\begin{proof}
	Since $H \subset T'$, we have the inclusion $\Omega^*(M)^{T'} \subset \Omega^*(M)^{H}$. 
	For $g \in T'$, take a sequence $\{g_i\}_{i =1,\dots}$ of elements in $H$ so that $\lim_{i \to \infty}g_i = g$. Then we have 
	\begin{equation*}
		g^{*}\omega =\lim_{i \to \infty} g_{i}^{*}\omega=\omega
	\end{equation*}
	for any $\omega \in \Omega^*(M)^{H}$, showing the opposite inclusion $\Omega^*(M)^{T'} \supset \Omega^*(M)^{H}$. The lemma is proved. 
\end{proof}

Consider the basic forms
\[\Omega^{*}_{B}(M)=\{\omega\in \Omega^{*}(M) \mid i_{X_{v}}\omega=L_{X_{v}}\omega=0,\, \forall v\in\h\}. 
\]

We want to construct a finite dimensional subspace $W \subset \Omega ^1(M)^H$ such that 
\begin{itemize}
\item $dW\subset \Omega^{2}_{B}(M)$,

\item the bilinear map $\h\times W\ni (v,w)\mapsto i_{X_v}w\in \R$ is non-degenerate.
\end{itemize}
To do this, take  a $\h$-valued $1$-form $\omega$ as in Lemma \ref{lemma:connection}. 
For a basis $v_1,\dots, v_k$ of $\h$, we may write \ishida{old: $w = \sum w_iv_i$} $w =\sum_{i=1}^k w_i\otimes v_i$ with $1$-forms $w_1,\dots, w_k$. We claim that  $dw_i \in \Omega_B^2(M)$. Since $w_i$ is $T$-invariant, by Cartan formula we have 
\begin{equation*}
	0 = L_{X_v} w_i = di_{X_v}w_i + i_{X_v}dw_i = i_{X_v}dw_i
\end{equation*}
for $v \in \h$ because $i_{X_v}w_i$ is constant on $M$. By Cartan formula again, 
\begin{equation*}
	L_{X_v}dw_i = di_{X_v}dw_i + i_{X_v}ddw_i = di_{X_v}dw_i.
\end{equation*}
This together with $i_{X_v}dw_i =0$ yields that $dw_i \in \Omega^*_B(M)$. Then $W=\langle w_{1},\dots, w_{l}\rangle$ is a desired space.

\begin{prop}\Label{HeqW}
We have the decomposition 
\[\Omega^{*}(M)^{H} = \Omega^*_{B}(M)\otimes \bigwedge W.
\]
\end{prop}
\begin{proof}
For $\omega\in \Omega^{*}_{B}(M)$, the condition $L_{X_v}\omega =$ for all $v \in \h$ implies  that $\omega\in \Omega^{*}(M)^H$.
Since $W\subset \Omega^{1}(M)^{H}$ and $\h\times W\ni (v,w)\mapsto i_{X_v}w\in \R$ is non-degenerate, we have the inclusion 
\[\Omega^{*}_{B}(M)\otimes \bigwedge W\subset \Omega^{*}(M)^{H}.
\]
We will show that $\Omega^{*}(M)^{H}\subset \Omega^{*}_{B}(M)\otimes \bigwedge W$.
We say that $\omega \in \Omega^{*}(M)^{H}$ is of $q$-type if for any $v_{1},\dots,v_{q}\in \h$ we have
\[i_{X_{v_{1}}}\dots i_{X_{v_{q}}}\omega=0.
\]
If $\omega \in \Omega^{*}(M)^{H}$ is of $1$-type,
 $\omega \in \Omega^{*}_{B}(M)$.
Suppose that  $\omega\in \Omega^{*}(M)^{H}$ is of $q$-type for some $q\ge 2$.
Then for any $v, v_1,\dots, v_{q-1} \in \h$, we have that
\[i_{X_{v}}i_{X_{v_{1}}}\dots i_{X_{v_{q-1}}}\omega=0
\]
and 
\[L_{X_{v}}i_{X_{v_{1}}}\dots i_{X_{v_{q-1}}}\omega =i_{X_{v_{1}}}\dots i_{X_{v_{q-1}}} L_{X_{v}} \omega=0.
\]
Therefore we have that 
\[i_{X_{v_{1}}}\dots i_{X_{v_{q-1}}}\omega\in\Omega^{*}_{B}(M).
\]
Take a basis $v_{1},\dots v_{k}$ of $\h$ and the dual basis 
$w_{1},\dots ,w_{k}$ of $W$ given by $W\subset \Omega^{1}(M)^{H}$ and $\h\times W\ni (v,w)\mapsto i_{X_v}w\in \R$.
Then for  $\omega\in \Omega^{*}(M)^{H}$  of $q$-type,
we can see that the form
\[\omega^{\prime}=\omega-\sum _{i_{1}<i_{2}<\dots<i_{q-1}}(i_{X_{v_{i_1}}}\dots i_{X_{v_{i_{q-1}}}}\omega)\wedge w_{i_{1}}\wedge\dots\wedge w_{i_{q-1}}.
\]
is of $(q-1)$-type. It turns out that $\omega-\omega' \in \Omega_B^*(M)\otimes \bigwedge^{q-1}W$. Since $\omega'$ is of $(q-1)$-type, applying the same argument eventually, we have that 
\[
\omega \in \bigoplus_{0 \leq j\leq q-1} \Omega_B^*(M)\otimes \bigwedge ^jW,
\]
showing the inclusion $\Omega^*(M)^H \subset \Omega_B^*(M)\otimes \bigwedge W$. The proposition is proved. 
\end{proof}

By Propositions \ref{cptderham}, \ref{HeqW} and Lemma \ref{lemmaGeqH}, we have the following result.

\begin{cor}\label{cor:tormodel}
The inclusion 
\[\Omega^{*}_{B}(M)\otimes \bigwedge W\to \Omega^{*}(M)
\]
induces \ishida{old:an} a cohomology isomorphism. 
\end{cor}

\begin{prop}\Label{prop:homori}
	Suppose that $\dim M = n+k$. Then, $H^{n+k}(M) \cong H_B^n(M)$. In particular, $\mathcal F$ is homologically oriented if $M$ is compact and oriented. 
\end{prop}
\begin{proof}
	By Proposition \ref{cptderham} and Lemma \ref{lemmaGeqH}, we can choose a representative $\alpha$ of an element in $H^{n+k}(M)$ so that $\alpha$ sits in $\Omega^{n+k}(M)^H$. By Proposition \ref{HeqW}, there uniquely exists $\beta \in \Omega_B^n(M)$ such that $\alpha = \beta \wedge w_1\wedge \dots \wedge w_k$. 
	Conversely, for $\beta \in \Omega_B^n(M)$, $\alpha := \beta\wedge  w_1\wedge \dots \wedge w_k \in \Omega^n(M)^H$. Thanks to the degrees, $\alpha$ and $\beta$ both are automatically closed. Therefore it suffices to show that $\alpha$ is exact if and only if $\beta$ is exact (in the sense of basic). 
	
	Let $\alpha' \in \Omega^{n+k-1}(M)^H$ such that $d\alpha' =\alpha$. By Proposition \ref{HeqW}, we can write 
	\begin{equation*}
		\alpha' = \beta' \wedge w_1\wedge \dots \wedge w_k + \sum_{i=1}^l\beta_i\wedge w_1 \wedge \dots \wedge \hat{w_i} \wedge \dots \wedge w_k
	\end{equation*}
	with $\beta' \in \Omega_B^{n-1}(M)$ and $\beta_i \in \Omega_B^n(M)$ for $i=1,\dots, l$. Then, it follows from $dw_j \in \Omega_B^2(M)$ that $\alpha = d\alpha' = d\beta'\wedge w_1\wedge \dots \wedge w_k$. In particular, $\beta = d\beta'$. 
	
	To see the converse, let $\beta' \in \Omega^{n-1}_B(M)$ such that $d\beta' = \beta$. Then 
	\begin{equation*}
		d(\beta' \wedge w_1\wedge \dots \wedge w_k) = \beta \wedge w_1\wedge \dots \wedge w_k + (-1)^{n-1}\beta' \wedge d(w_1\wedge \dots \wedge w_k). 
	\end{equation*}
	Since $\beta' \wedge dw_j =0$ by the degree, we have that 
	\begin{equation*}
		\alpha = d (\beta' \wedge w_1\wedge \dots \wedge w_k), 
	\end{equation*}
	showing the equivalence of exactness between $\alpha$ and $\beta$. The proposition is proved. 
\end{proof}

\subsection{Models for transverse K\"ahler torus actions}
Let $M$ be a compact complex manifold and $T$ a compact torus acting on $M$ as biholomorphisms. 
Let $H$ be a dense Lie subgroup of $T$ such that $H$ is a complex Lie group and
the restricted action of $T$ to $H$ on $M$ is holomorphic and local free. 
Then we have a holomorphic central foliation $\mathcal F$ on $M$ whose leaves are $H$-orbits. As before, let $\Omega^*_B(M)$ denote the space of basic differential forms with respect to $\mathcal F$. Since $M$ is a complex manifold and the $H$-action is holomorphic, $\Omega^1(M)^H$ and $\Omega^1_B(M)$ both are complex vector spaces. 
\begin{prop}\Label{prop:W} There exists a $\C$-subspace $W$ of $\Omega^1(M)^H$ such that 
\begin{itemize}
	\item $dW \subset \Omega^2_B(M)$ and 
	\item $\h \times W \ni (v,w) \mapsto i_{X_v}w \in \R$ is non-degenerate.
\end{itemize}
\end{prop}
\begin{proof}
	By Lemma \ref{lemma:connection}, there exists a $\mathfrak{h}$-valued $1$-form $w \in \Omega^1(M)\otimes \mathfrak{h}$  on $M$ such that $i_{X_v}w = v $ for all $v \in \mathfrak{h}$ and $H$-invariant. Let $v_1,\dots, v_k$ be a $\C$-basis of $\mathfrak{h}$ and $J_\mathfrak{h}$ the complex structure on $\mathfrak{h}$. Then $v_1,\dots, v_k, J_\mathfrak{h}v_1,\dots, J_\mathfrak{h}v_k$ form a $\R$-basis of $\mathfrak{h}$. There exist $w_1,\dots, w_k, w_{k+1},\dots, w_{2k} \in \Omega^1(M)^H$ such that 
	\begin{equation*}
		w = \sum_{i=1}^k w_i\otimes v_i + \sum_{j=1}^k w_{k+j}\otimes J_\mathfrak{h}v_j. 
	\end{equation*}
	For $i=1,\dots, k$, we define $w_i' \in \Omega^1(M)^H$ to be $w_i' = -w_i\circ J$, where $J$ denotes the complex structure on $M$. We define an $H$-invariant $\mathfrak{h}$-valued $1$-form 
	\begin{equation*}
		w' = \sum_{i=1}^k w_i\otimes v_i + \sum_{j=1}^k w_j' \otimes J_\mathfrak{h}v_j. 
	\end{equation*}
	It follows that $i_{X_v}w' = v$ for all $v \in \mathfrak{h}$ by definition of $w'$ immediately. 
	The subspace $W= \langle w_1,\dots, w_k, w_1', \dots, w_k'\rangle$ of $\Omega^1(M)^H$ is closed under $J$. It follows from the Cartan formula that $dW \subset \Omega_B^2(M)$ immediately. Therefore $W$ is a desired space, proving the proposition. 
\end{proof}
\ishida{old: It follows from Proposition \ref{HeqW} that there exists a finite dimensional complement $W$ of $\Omega^1_B(M)$ in $\Omega^1(M)^H$. Remark that any complement $W$ satisfies the conditions 
\begin{itemize}
	\item $dW \subset \Omega^2_B(M)$ and 
	\item $\h \times W \ni (v,w) \mapsto i_{X_v}w \in \R$ is non-degenerate.
\end{itemize}}
\ishida{old: Since the codimension of $\Omega^1_B(M)$ in $\Omega^1(M)^H$ is finite, there exists a complement $W$ which is closed under the complex structure. Then $W\otimes \C$ is decomposed into $(1,0)$-part $W^{1,0}$ and $(0,1)$-part $W^{0,1}$.}
Let $W$ be a $\C$-subspace of $\Omega^1_B(M)$ as in Proposition \ref{prop:W}. Then $W\otimes \C$ is decomposed into $(1,0)$-part $W^{1,0}$ and $(0,1)$-part $W^{0,1}$. By tensoring $\C$ with $\Omega^*(M)\otimes \bigwedge W$, we have the DBA 
\[
\Omega^{*,*}_B(M)\otimes\bigwedge (W^{1,0}\oplus W^{0,1})
\]
with the Dolbeault operator $\bar\partial$.

By  Propositions \ref{prop:torusdol}, \ref{HeqW} and Lemma \ref{lemmaGeqH}, we have the following result.

\begin{cor}\label{cor:dol}
We have an injection  
\[\Omega^{*,*}_{B}(M)\otimes \bigwedge (W^{1,0}\oplus W^{0,1})\to \Omega^{*,*}(M)
\]
which induces a cohomology isomorphism. 
\end{cor}

We consider the bi-graded bi-differential algebra (BBA) $(\Omega^{*,*}_{B}(M),\partial_{B},\bar\partial_{B})$.
Put $d^{c}=\sqrt{-1}(\bar\partial_B-\partial_B)$.
Then $d^{c}$ is a differential on $\Omega^{*}_{B}(M)$.
We say that  the $\partial_{B}\bar\partial_{B}$-lemma holds if
\[\ker \partial_{B}\cap \ker\bar\partial_{B}\cap \operatorname{im} d=\operatorname{im}  \partial_{B}\bar\partial_{B}.
\]
If the $\partial_{B}\bar\partial_{B}$-lemma holds, then 
we have the quasi-isomorphisms
\[(\ker d^{c}, d)\to (\Omega^{*}_{B}(M),d),
\]
\[(\ker d^{c}, d)\to (H_{B}^{*}(M),0),
\]
\[(\ker \partial_{B}, \bar\partial_{B})\to (\Omega^{*,*}_{B}(M),\bar\partial_{B})
\]
and 
\[(\ker \partial_{B}, \bar\partial_{B})\to (H_{B}^{*,*}(M),0)
\]
(see \cite{DGMS}).

\begin{prop}\Label{prop:quasiiso}
Suppose that   the $\partial_{B}\bar\partial_{B}$-lemma holds.
Then there exist quasi-isomorphisms
\[(\ker d^{c}\otimes \bigwedge W, d)\to (\Omega^{*}_{B}(M)\otimes \bigwedge W,d),
\]
\[(\ker d^{c}\otimes \bigwedge W, d)\to (H_{B}^{*}(M)\otimes \bigwedge W,d),
\]
\[(\ker \partial_B\otimes \bigwedge (W^{1,0}\oplus W^{0,1}), \bar\partial')\to (\Omega^{*,*}_{B}(M)\otimes \bigwedge (W^{1,0}\oplus W^{0,1}),\bar\partial)
\]
and 
\[(\ker \partial_{B}\otimes \bigwedge (W^{1,0}\oplus W^{0,1}), \bar\partial')\to (H_{B}^{*,*}(M)\otimes \bigwedge (W^{1,0}\oplus W^{0,1}),\bar\partial).
\]
Here, $\bar \partial'$ is a differential such that $(\bar \partial' -\bar\partial)w$ is $\partial _B$-exact for any $w \in W^{1,0}\oplus W^{0,1}$ and $\bar \partial'\alpha = \bar\partial\alpha$ for any $\alpha \in \ker \partial_B$.
\end{prop}

\begin{proof}
This follows from Lemma \ref{hir} immediately.

\end{proof}

\begin{thm}[see \cite{DGMS} and \cite{EKA}]
Let $M$ be a compact manifold with a homologically oriented (that is, $H_{B}^{\operatorname{codim}\mathcal F}(M)\not=0$) transversely K\"ahler foliation $\mathcal F$.
Then for the BBA $(\Omega^{*,*}_{B}(M),\partial_{B},\bar\partial_{B})$, the $\partial_{B}\bar\partial_{B}$-lemma holds.
\end{thm}

This together with Propositions \ref{prop:homori} and \ref{prop:quasiiso} implies the following result. 
\begin{thm}\label{theorem-model}

Assume that the central foliation
$\mathcal F$ admits a transversely K\"ahler structure.
Then the DGAs $\Omega^{*}(M)$ and   $H_{B}^{*}(M)\otimes \bigwedge W$ (resp.\ DBAs $\Omega^{* ,*}(M)$ and $H_{B}^{*,*}(M)\otimes \bigwedge (W^{1,0}\oplus W^{0,1})$) are quasi-isomorphic.
\end{thm}

\section{Mixed Hodge structures}
The purpose of this section is to show that the cohomology and minimal model of a complex manifold equipped with a special transverse K\"ahler structure on a central foliation admits a certain bigrading. We begin with basic notions and facts. 
\subsection{Mixed Hodge structures}
Let $V$ be  an $\R$-vector space.
An {\em$\R$-Hodge structure} of weight $n$ on an $\R$-vector space $V$ is a finite decreasing filtration $F^{\ast}$ on $V_\C = V \otimes \C$
such that 
\[F^{p}(V_{\C})\oplus \overline {F^{n+1-p}(V_{\C})}=V_{\C}
\]
for each $p$.
Equivalently, there exists a 
 finite bigrading 
\[V_{\C}=\bigoplus_{p+q=n}V_{p,q}
\]
such that
\[\overline{V_{p,q}}=V_{q,p}.
\]
An {\em$\R$-mixed-Hodge structure} on  $V$ is a pair $(W_{\ast},F^{\ast})$
such that:
\begin{enumerate}
\item $W_{\ast}$ is an increasing filtration which is bounded below,
\item $F^{\ast}$ is a decreasing filtration  on $V_{\C}$ such that
the filtration on $Gr_{n}^{W} V_{\C}$ induced by $F^{\ast}$ is an $\R$-Hodge structure of weight $n$.
\end{enumerate}
We call $W_{\ast}$ the {\em weight filtration} and $F^{\ast}$ the {\em Hodge filtration}.
If there exists a finite
 bigrading 
\[V_{\C}=\bigoplus V_{p,q} 
\]
satisfying 
\[\overline{V_{p,q}}=V_{q,p}, 
\]
 $W_{n}(V_{\C})=\bigoplus_{p+q\le n} V_{p,q} $
and $F_{r}(V_{\C})=\bigoplus_{p\ge r} V_{p,q} $ for any $n,p,q,r$,
then we say that  an $\R$-mixed-Hodge structure $(W_{\ast},F^{\ast})$ is $\R$-split.

Even if an $\R$-mixed-Hodge structure $(W_{\ast},F^{\ast})$ is not $\R$-split, we can obtain a canonical bigrading of $(W_{\ast},F^{\ast})$.

\begin{prop}\label{BIGG}{\rm (\cite[Proposition 1.9]{Mor})}
Let $(W_{\ast},F^{\ast})$ be an $\R$-mixed-Hodge structure on an $\R$-vector space $V$.
Define $V_{p,q}=R_{p,q}\cap L_{p,q}$ where
$R_{p,q}=W_{p+q}(V_{\C})\cap F^{p}(V_{\C})$ and
$L_{p,q}=W_{p+q}(V_{\C})\cap \overline{F^{q}(V_{\C})}+\sum_{i\ge 2} W_{p+q-i}(V_{\C})\cap \overline{F^{q-i+1}(V_{\C})}$.
Then we have 
  the   bigrading $V_{\C}=\bigoplus V_{p,q}$ such that
 $\overline{V_{p,q}}=V_{q,p}$ modulo $\bigoplus_{r+s<p+q} V_{r,s}$, 
$W_{n}(V_{\C})=\bigoplus_{p+q\le n}V_{p,q}$ and
$F^{r}(V_{\C})=\bigoplus_{p\ge r} V_{p,q}$.

\end{prop}
We say that the bigrading in this proposition is the {\em canonical bigrading} of  an $\R$-mixed-Hodge structure $(W_{\ast},F^{\ast})$.

We notice that this bigrading gives an equivalence of  the category of $\R$-mixed-Hodge structures on $V$ and bigradings $V_{\C}=\bigoplus V_{p,q}$ such that $\left(\bigoplus_{p+q\le i}V_{p,q}\right)\cap V$ is a real  structure of  $\bigoplus_{p+q\le i}V_{p,q}$ and $\overline{V_{p,q}}=V_{q,p}$ modulo $\bigoplus_{r+s<p+q} V_{r,s}$ (see \cite[Proposition 1.11]{Mor}).

\subsection{Morgan's Mixed Hodge diagrams}
In \cite{Del}, Deligne proves that  the real cohomology of a smooth algebraic variety over $\C$ admits a canonical $\R$-mixed-Hodge structure.
The following is Morgan's reformulation of Deligne's technique for studying the mixed Hodge theory  on Sullivan's minimal models.

\begin{defi}[{\cite[Definition 3.5]{Mor}}]
An $\R$-{\em mixed-Hodge diagram} is a pair of filtered $\R$-DGA $(A^{*}, W_{*})$ and bifiltered $\C$-DGA $(E^{*}, W_{*},F^{*})$ and filtered DGA map $\phi:(A^{*}_{\C},W_{*})\to (E^{*},W_{*})$  such that:
\begin{enumerate}
\item $\phi$  induces an isomorphism  $\phi^{*}:\,_{W}E^{*,*}_{1}(A^{*}_{\C})\to \,_{W}E^{*,*}_{1}(E^{*})$ where $ \,_{W}E_{*}^{*,*}(\cdot)$ is the spectral sequence for the decreasing filtration $W^{*}=W_{-*}$.
\item The differential $d_{0}$ on $\,_{W}E^{*,*}_{0}(E^{*})$ is strictly compatible with the filtration induced by $F$.
\item The filtration on $\,_{W}E_{1}^{p,q}(E^{*})$ induced by $F$ is an $\R$-Hodge structure of weight $q$ on $\phi^{*}(\,_{W}E^{*,*}_{1}(A^{*}))$.

\end{enumerate}
\end{defi}

Now, Deligne's $\R$-mixed Hodge structure is described by the following way.

\begin{thm}[{\cite[Theorem 4.3]{Mor}}]\label{midimi}
Let $\{(A^{*}, W_{*}), (E^{*}, W_{*},F^{*}),\phi\}$ be an $\R$-mixed-Hodge diagram.
Define the filtration $W^{\prime}_{*}$ on $H^{r}(A^{*})$ (resp.\ $H^{r}(E^{*}))$
as $W^{\prime}_{i}H^{r}(A^{*})=W_{i-r}(H^{r}(A^{*}))$ (resp.\ $W^{\prime}_{i}H^{r}(E^{*})=W_{i-r}(H^{r}(E^{*}))$).
Then the filtrations $W^{\prime}_{*}$ and $F^{*}$ on $H^{r}(E^{*})$ give an $\R$-mixed-Hodge  on 
$\phi^{*}(H^{r}(A^{*}))$.
\end{thm}

\begin{ex}\label{ex:mho}
Let $H^{*}$ be a graded commutative $\R$-algebra.
We suppose that for any $p,q $, $H^{p}$ admits an $\R$-Hodge structure $H^{p}\otimes \C=\bigoplus_{s+t=p} H^{s,t}$ of weight $p$ and  the multiplication $H^{p}\times H^{q}\to H^{p+q}$ is a morphism of Hodge structures.
Let $V$ be an $\R$-vector space with a linear map $\beta:V\to H^{2}$.
We suppose that $V$  admits an $\R$-Hodge structure $V\otimes \C=\bigoplus_{s+t=2} V^{s,t}$ of weight $2$  and 
 $\beta:V\to H^{2}$ is a morphism of Hodge structure. (e.g.\ $ \beta(V)\subset H^{1,1}$.)

Under  these assumptions, 
regarding $H^{*}$ as a DGA with trivial differential,
we consider the Hirsch extension $A^{*}=H^{*}\otimes \bigwedge V$.
Define the increasing filtration $W_{*}A^{*}$ as
\[W_{k}A^{q}=\bigoplus_{l\le k} H^{q-l}\otimes \bigwedge^{l} V
\]
and decreasing filtration $F^{*}A^{*}_{\C}$ as the Hodge filtration for the Hodge structure on $(H^{p}\otimes \C)\otimes \bigwedge^{q} (V\otimes \C)$.
Then
for any $p,q$, we have:
\begin{itemize}
\item $\,_{W}E^{-p,q}_{0}(A^{*}_{\C})=(H^{q-2p}\otimes \C)\otimes \bigwedge^{p} (V\otimes \C)$ and $d_{0}$ is trivial.
\item $\,_{W}E^{-p,q}_{1}(A^{*}_{\C})=(H^{q-2p}\otimes \C)\otimes \bigwedge^{p} (V\otimes \C)$ and clearly $F$ induces the Hodge structure of weight $q$.
\end{itemize}
Thus $\{(A^{*},W_{*}), (A^{*}_{\C},W_{*},F^{*}), \operatorname{id} :A^{*}_{\C}\to A^{*}_{\C}\}$ is an $\R$-mixed-Hodge diagram.

We can easily check that for  the canonical bigrading $H^{r}(A^{*}_{\C})=\bigoplus H^{r}_{p,q}$
 of the $\R$-mixed Hodge structure as in Theorem \ref{midimi}, we have
\begin{enumerate}
\item $H^{1}(A^{*}_{\C})=H^{1}_{1,0}\oplus H^{1}_{0,1}\oplus H_{1,1}^{1}$.
\item $H^{2}(A^{*}_{\C})=H^{2}_{2,0}\oplus H^{2}_{1,1}\oplus H_{0,2}^{2}\oplus H^{2}_{2,1}\oplus H^{2}_{1,2}\oplus H_{2,2}^{2}$.

\end{enumerate}
\end{ex}

Morgan's result on Sullivan's minimal models of $\R$-mixed-Hodge diagrams  is following.

\begin{thm}[{\cite[Section 6, 8]{Mor})}]\label{MMM11}
Let $\{(A^{*}, W_{*}), (E^{*}, W_{*},F^{*}),\phi\}$ be an $\R$-mixed-Hodge diagram.
Then the minimal model (resp.\ $1$-minimal model) $\mathcal M^{*}$  of the DGA $E^{*}$ with a quasi-isomorphism (resp.\ $1$-quasi-isomorphism) $\phi:\mathcal M^{*}\to E^{*}$ 
satisfies the following conditions:
\begin{itemize}
\item $\mathcal M^{*}$ admits a bigrading 
\[\mathcal M^{*}=\bigoplus_{p,q\ge0}\mathcal M^{*}_{p,q}
\]
such that $\mathcal M^{*}_{0,0}=\mathcal M^{0}=\C$ and 
the product and the differential are of type $(0,0)$.

\item For some real structure of $\mathcal M^{*}$,  the bigrading $\bigoplus_{p,q\ge0}\mathcal M^{*}_{p,q}$ induces an $\R$-mixed-Hodge structure.

\item 
Consider the canonical bigrading $H^{r}(E^{*})=\bigoplus V_{p,q}$ for the  $\R$-mixed-Hodge structure  as in Theorem \ref{midimi}.
Then  $\phi^{*}:H^{r}(\mathcal M^{*})\to H^{r}(E^{*})$ 
sends $H^{r}(\mathcal M^{*}_{p,q})$ to $V_{p,q}$.
\end{itemize}
\end{thm}

\subsection{Mixed Hodge diagrams for transverse K\"ahler structures on central foliations}
Let  $M$ be a compact complex manifold.
We assume that $M$ admits a   transverse K\"ahler structure on a  central foliation ${\mathcal  F}_{H}$.
Let $(\Omega^{*,*}_{B}(M),\partial_{B},\bar\partial_{B})$ be the BBA of basic differential forms associated with $\mathcal{F}_H$. 
The \emph{basic Bott-Chern cohomology} $H^{\ast,\ast}_{B,BC}(M)$ is defined to be \[H^{\ast,\ast}_{B, BC}(M)=\frac{{\rm Ker}\,\partial_{B}\cap {\rm Ker}\,\bar\partial_{B}}{{\rm Im}\,\partial_{B}\bar\partial_{B}}.
\]
Then we have $\overline{H^{p,q}_{B, BC}(M)}=H^{q,p}_{B, BC}(M)$ and 
the natural algebra homomorphisms
\[{\rm Tot} ^{\ast}H^{\ast,\ast}_{B, BC}(M)\to H^{\ast}_{B}(M,\C)
\]
and
\[H^{\ast,\ast}_{B, BC}(M)\to H^{\ast,\ast}_{B}(M).
\]
By  $\partial_{B}\bar\partial_{B}$-Lemma, these maps  are isomorphisms (see \cite[Remark 5.16]{DGMS}).
Thus, we have the Hodge decomposition 
\[H^{r}_{B}(M,\C)=\bigoplus_{p+q=r}H^{p,q}_{B}(M)
\]
and 
\[\overline{H^{p,q}_{B}(M)}=H^{q,p}_{B}(M).
\]
We remark that  this decomposition  does not depend on the choice of a  transverse K\"ahler structure.

Under the assumptions as in Theorem \ref{theorem-model},
we consider the model ${\mathcal A}^{*}=H_{B}^{*}(M)\otimes \bigwedge W$  as in Theorem \ref{theorem-model}.
We suppose that  ${\mathcal  F}_{H}$ is fundamental as in Definition \ref{speci}.
we can obtain the mixed Hodge diagram $\{(A^{*},W_{*}), (A^{*}_{\C},W_{*},F^{*}), \operatorname{id} :A^{*}_{\C}\to A^{*}_{\C}\}$ as in Example \ref{ex:mho}.
Finally we obtain the following statement.

\begin{thm}\label{theorem-MHS}

Let  $M$ be a compact complex manifold.
We assume that $M$ admits a    transverse K\"ahler structure on a  fundamental central foliation ${\mathcal  F}_{H}$.
Consider the minimal model $\mathcal M$ (resp.\ $1$-minimal model) of $A^{*}_{\C}(M)$ with a quasi-isomorphism  (resp.\ $1$-quasi-isomorphism) $\phi:\mathcal M\to A_{\C}^{*}(M)$.
Then we have:
\begin{enumerate}
\item For each $r$,  the real de Rham cohomology $H^{r}(M,\R)$ admits an 
$\R$-mixed-Hodge structure such that
\begin{itemize}
\item $H^{1}(M,\C)=H^{1}_{1,0}\oplus H^{1}_{0,1}\oplus H_{1,1}^{1}$
\item $H^{2}(M,\C)=H^{2}_{2,0}\oplus H^{2}_{1,1}\oplus H_{0,2}^{2}\oplus H^{2}_{2,1}\oplus H^{2}_{1,2}\oplus H_{2,2}^{2}$
\end{itemize}
where $H^{r}(M,\C)=\bigoplus  H^{r}_{p,q}$ is the canonical bigrading.
\item 
$\mathcal M^{*}$ admits a bigrading 
\[\mathcal M^{*}=\bigoplus_{p,q\ge0}\mathcal M^{*}_{p,q}
\]
such that $\mathcal M^{*}_{0,0}=\mathcal M^{0}=\C$ and 
the product and the differential are of type $(0,0)$.

\item For some real structure of $\mathcal M^{*}$,  the bigrading $\bigoplus_{p,q\ge0}\mathcal M^{*}_{p,q}$ induces an $\R$-mixed-Hodge structure.
\item
The induced map $\phi^{*}:H^{r}(\mathcal M^{*})\to H^{r}(M,\C)$ 
sends $H^{r}(\mathcal M^{*}_{p,q})$ to $H^{r}_{p,q}$.

\end{enumerate}
\end{thm}
In this theorem,
 for the $1$-minimal model $\mathcal M$  with a  $1$-quasi-isomorphism $\phi:\mathcal M\to A_{\C}^{*}(M)$,
we have:
\begin{itemize}
\item $H^{1}(\mathcal M^{*})=H^{1}(\mathcal M^{*}_{1,0})\oplus H^{1}(\mathcal M^{*}_{0,1})\oplus H^{1}(\mathcal M^{*}_{1,1})$
\item $H^{2}(\mathcal M^{*})=H^{2}(\mathcal M^{*}_{2,0})\oplus H^{2}(\mathcal M^{*}_{1,1})\oplus H^{2}(\mathcal M^{*}_{0,2})\oplus H^{2}(\mathcal M^{*}_{2,1})\oplus H^{2}(\mathcal M^{*}_{1,2})\oplus H^{2}(\mathcal M^{*}_{2,2})$.
\end{itemize}

By Theorem \ref{dHom}, we can translate this condition to certain condition on  the Lie algebra of the nilpotent completion of the fundamental group $\pi_{1}(M)$ as \cite[Theorem 9.4]{Mor} .
We obtain:
\begin{thm}
Let  $M$ be a compact complex manifold.
We assume that $M$ admits a   transverse K\"ahler structure on a fundamental central foliation ${\mathcal  F}_{H}$.
Then the the Lie algebra of the nilpotent completion of the fundamental group $\pi_{1}(M)$ is isomorphic to ${\mathcal F}(H)/{\mathcal I}$ such that
\begin{itemize}
\item $H$ is a $\C$-vector space with a bigrading $H=H_{-1,0}\oplus H_{0,-1}\oplus H_{-1,-1}$
\item ${\mathcal I}$ is a Homogeneous ideal of the free bigraded  Lie algebra generated by $H$  such that ${\mathcal I}$ has generators of types $(-1,-1)$, $(-1,-2)$, $(-2,-1)$ and $(-2,-2)$ only.
\end{itemize}
\end{thm}
As a consequence,    the Lie algebra of the nilpotent completion of the fundamental group $\pi_{1}(M)$ is determined by $\pi_{1}(M)/\Gamma_{5}$ where $\Gamma_{5}$ is the 5th-term of the lower central series of  $\pi_{1}(M)$ (\cite[Corollary 9.5]{Mor}).
Thus, we can say that  not every finitely generated group can be the fundamental  group of a compact complex  manifold with    transverse K\"ahler structure on a fundamental  central foliation.

\section{Examples and applications}
\subsection{Simple examples}
\begin{ex}\label{S1n}
Consider the product $S^{1,2n-1}=S^{1}\times S^{2n-1}$ of a circle and  a $(2n-1)$-dimensional sphere 
equipped with a complex structure so that there exists a  special  transverse K\"ahler structure on a $1$-dimensional central foliation ${\mathcal  F}_{H}$.
Then, by our results, $\Omega^{*}(S^{1,2n-1})$ is quasi-isomorphic to the DGA $ A^{*}=H_{B}^{*}(S^{1,2n-1})\otimes \bigwedge W$.
By $\dim H^{1}(S^{1,2n-1})=1$ and $H^{1}(S^{1,2n-1},\C)=H^{1,0}_{B}(S^{1,2n-1})\oplus H^{0,1}_{B}(S^{1,2n-1})\oplus \ker d\vert _{W}$,
we have $H^{1,0}_{B}(S^{1,2n-1})\oplus H^{0,1}_{B}(S^{1,2n-1})=0$ and $\dim \ker d\vert _{W}=1$.
By $\dim H^{2}(S^{1,2n-1})=0$, the differential $d:W\to H^{2}_{B}(S^{1,2n-1})$ is surjective and hence $\dim H^{2}_{B}(S^{1,2n-1})=1$. 
Take $W=\langle x,y\rangle $ so that $dx\not=0$ in $H^{2}_{B}(S^{1,2n-1})$ and $dy=0$.
We have $H^{2}_{B}(S^{1,2n-1})=\langle dx\rangle$.
Since $ dx\in H^{2}_{B}(S^{1,2n-1})$ must contain transverse K\"ahler form, we have $(dx)^{i}\not=0$ for any $i\le n-1$.
Inductively we can easily compute $ H^{2i}_{B}(S^{1,2n-1})=\langle (dx)^{i}\rangle$ and $ H^{2i-1}_{B}(S^{1,2n-1})=0$ for $2\le i\le n-1$.

Consider the Hodge decomposition 
\[H^{r}_{B}(S^{1,2n-1},\C)=\bigoplus_{p+q=r}H^{p,q}_{B}(S^{1,2n-1}).
\]
Then we have $H^{i,i}_{B}(S^{1,2n-1})=\langle (dx)^{i}\rangle$
 for any $i\le n-1$ and $H^{p,q}_{B}(S^{1,2n-1}=0$ for $p\not=q$.
 Take the decomposition $W\otimes \C=W^{1,0}\oplus W^{0,1}$
with $W^{1,0}=\langle z\rangle$.
Then we have $dz=cdx $ for some $c\in\C$.
Thus we have $\bar\partial z= cdx $ and $\bar\partial \bar z= 0 $.
Hence $\Omega^{* ,*}(S^{1,2n-1})$  is quasi-isomorphic to the DBA
\[B^{*,*}=\langle 1, dx,\dots, (dx)^{n-1}  \rangle \otimes \bigwedge \langle z, \bar z\rangle.
\]
Thus every complex structure on $S^{1,2n-1}$ with a   transverse K\"ahler structure on a $1$-dimensional fundamental central foliation ${\mathcal  F}_{H}$ has same basic Betti, basic Hodge and Hodge numbers.
There are many such complex structures, see Example \ref{hop}.
\end{ex}

\begin{ex}
Consider the product $S^{3,3}=S^{3}\times S^{3}$ of two three dimensional spheres
equipped with a complex structure so that there exists a    transverse K\"ahler structure on a $1$-dimensional central foliation ${\mathcal  F}_{H}$.
Then, by our results, $\Omega^{*}(S^{3,3})$ is quasi-isomorphic to the DGA $ A^{*}=H_{B}^{*}(S^{3,3})\otimes \bigwedge W$.
By $H^{1}(S^{3,3})=0$ and $H^{2}(S^{3,3})=0$,
 we have $H^{1}_{B}(S^{3,3})=0$ and the differential $d:W\to H^{2}_{B}(S^{3,3})$ is bijective.
Take $W=\langle x,y\rangle$.
Then $H^{2}_{B}(S^{3,3})=\langle dx, dy\rangle$.
By $\dim H^{3}(S^{3,3})=2$, just two of the elements
\[d(x\wedge dx)=dx\wedge dx, d(y\wedge dy)=dy\wedge dy, d(x\wedge dy)=-d(y\wedge dx)=dx\wedge dy 
\]
are equal to $0$. 
Take $x,y$ so that $dx\wedge dx=dy\wedge dy=0$ and $dx\wedge dy\not=0$.
Since the codimension of ${\mathcal  F}_{H}$ is $4$,
we have $\dim H^{4}_{B}(S^{3,3})=1$ and thus $H^{4}_{B}(S^{3,3})= \langle dx\wedge dy\rangle $.
Thus we have 
\[H_{B}^{*}(S^{3,3})=\bigwedge \langle dx, dy \rangle=  \langle 1,  dx, dy , dx\wedge dy\rangle.\]

Consider the Hodge decomposition 
\[H^{r}_{B}(S^{3,3},\C)=\bigoplus_{p+q=r}H^{p,q}_{B}(S^{3,3}).
\]
Then, by $H^{1,1}_{B}(S^{3,3})\not=0$ and $\dim H^{2}_{B}(S^{3,3})=2$,
we have that $H^{2,0}_{B}(S^{3,3})=H^{0,2}_{B}(S^{3,3})=0$.
Thus $H^{1,1}_{B}(S^{3,3})=\C\langle dx, dy\rangle$.
 Take the decomposition $W\otimes \C=W^{1,0}\oplus W^{0,1}$
with $W^{1,0}=\langle \alpha+\sqrt{-1}\beta\rangle$.
Now we have
\[\bar\partial(\alpha+\sqrt{-1}\beta)=d\alpha+\sqrt{-1}d\beta.
\]
and
\[\bar\partial(\alpha-\sqrt{-1}\beta)=0.\]
By $\langle x, y\rangle=\langle \alpha, \beta\rangle$,
we have 
\[d\alpha\wedge d\beta\not=0\in H^{4}_{B}(S^{3,3},\C)=H^{2,2}_{B}(S^{3,3}).
\]
Hence $\Omega^{* ,*}(S^{3,3})$  is quasi-isomorphic to the DBA
\[B^{*,*}=\langle 1, d\alpha, d\beta, \alpha\wedge d\beta\rangle \otimes \bigwedge \langle \alpha+\sqrt{-1}\beta, \alpha-\sqrt{-1}\beta\rangle.
\]
We compute
\[H^{1,0}(S^{3,3})=H^{2,0}(S^{3,3})=H^{3,0}(S^{3,3})=H^{0,2}(S^{3,3})=H^{0,3}(S^{3,3})=0
\]
and
\[\dim H^{0,1}(S^{3,3})= \dim H^{2,1}(S^{3,3})=\dim H^{1,2}(S^{3,3})=1.
\]
Thus every complex structure on $S^{3,3}$ with a   transverse K\"ahler structure on a $1$-dimensional  central foliation ${\mathcal  F}_{H}$ has same basic Betti, basic Hodge and Hodge numbers.
Such complex manifolds are constructed as LVM manifolds associated with complex numbers $(\lambda_{1},\dots,\lambda_{5}) $ with certain conditions (see \cite[Section 5]{MV}).
\end{ex}
\begin{ex}
	Consider the product $S^{1,3} = S^1 \times S^3$ (resp.~$S^{3,3} = S^3 \times S^3$) equipped with a complex structure so that there exists a  transverse K\"ahler structure on a $1$-dimensional central foliation $\mathcal{F}_{H_1}$ (resp.~$\mathcal{F}_{H_2}$). Then the product $S^{1,3} \times S^{1,3}$ has the natural complex structure so that there exists a special transverse K\"ahler structure on a $2$-dimensional central foliation $\mathcal{F}_{H_1 \times H_1}$. The K\"unneth formula allows us to compute the basic Betti, basic Hodge and Hodge numbers. By K\"unneth formula we have 
	\begin{equation*}
		\dim H^i_B (S^{1,3} \times S^{1,3}) =
		\begin{cases}
			1 & i =0, 4,\\
			2 & i =2,\\
			0 & \text{otherwise}, 
		\end{cases}
	\end{equation*}
	\begin{equation*}
		\dim H^{p,q}_B (S^{1,3} \times S^{1,3}) =
		\begin{cases}
			1 & p =q = 0,2,\\
			2 & p =q=1,\\
			0 & \text{otherwise} 
		\end{cases}
	\end{equation*}
	and 
	\begin{equation*}
		\dim H^{p,q} (S^{1,3} \times S^{1,3}) =
		\begin{cases}
			1 & (p,q) =(0,0),(4,4),(0,2),(4,2),\\
			2 & (p,q) = (0,1), (4,3), (1,2), (3,2),\\
			4 & (p,q) = (2,2), \\
			0 & \text{otherwise}. 
		\end{cases}
	\end{equation*}
	Now we consider the complex $1$-dimensional torus $S^{1,1} = S^1 \times S^1$ and the central foliation $\mathcal{F}_{S^{1,1}}$ on $S^{1,1}$. Then the product $S^{1,1} \times S^{3,3}$ has the natural complex structure so that there exists a  transverse K\"ahler structure on a $2$-dimensional central foliation $\mathcal{F}_{S^{1,1} \times H_2}$. By K\"unneth formula we have 
	\begin{equation*}
		\begin{split}
		\dim H^i_B (S^{1,1} \times S^{3,3}) &=
		\begin{cases}
			1 & i =0, 4,\\
			2 & i =2,\\
			0 & \text{otherwise}
		\end{cases}
		\\
		&= \dim H^i_B (S^{1,3} \times S^{1,3}), 
		\end{split}
	\end{equation*}
	\begin{equation*}
		\begin{split}
		\dim H^{p,q}_B (S^{1,1} \times S^{3,3}) &=
		\begin{cases}
			1 & p =q = 0,2,\\
			2 & p =q=1,\\
			0 & \text{otherwise} 
		\end{cases}\\
		&= \dim H^{p,q}_B (S^{1,3} \times S^{1,3})
		\end{split}
	\end{equation*}
	but
	\begin{equation*}
		\dim H^{p,q} (S^{1,1} \times S^{3,3}) \neq \dim H^{p,q}(S^{1,3} \times S^{1,3})
	\end{equation*}
	for some $p,q$. 
	Indeed, $\dim H^{1,0} (S^{1,1} \times S^{3,3}) = 1$ but $ \dim H^{1,0}(S^{1,3} \times S^{1,3}) = 0$. Thus, in general, the Hodge numbers depend on a complex structure.
\end{ex}

\subsection{Nilmanifolds}
Let $N$ be a simply connected nilpotent Lie group.
We suppose that $N$ admits a lattice $\Gamma$ i.e.\ cocompact discrete subgroup.
A compact homogeneous space $\Gamma\backslash N$ is called a {\em nilmanifold}.
It is known that a nilmanifold admits a K\"ahler structure if and only if it is a torus (see \cite{BG,hasegawa}).

Denote by $\frak n$  the Lie algebra of $N$.
Let $J$ be  an endomorphism  of $\frak n$ satisfying $J\circ J=-\operatorname{id}$ and $[JA, JB]=[A,B]$ for any $A, B\in \frak n$.
Then $J$ induces a complex structure on $\Gamma\backslash N$.
Such complex structure is called {\em abelian}. 
We assume that $\frak n$ is non-abelian and $2$-step i.e. $[\frak n,[\frak n,\frak n]]=0$.
Let $C$ be the center of $N$ and $\psi:N\to N/C$ the quotient map.
Then we have the holomorphic principal torus bundle
\[T\hookrightarrow \Gamma\backslash N\to M
\]
where $T$ and $M$ are complex tori $ \Gamma\cap C\backslash C$ and  $M=\psi(\Gamma)\backslash \psi(N)$ respectively.
Let $\frak c$ be the sub-algebra of $\frak n$ corresponding to $C$.
Consider the complex $\bigwedge {\frak n}^{*}$ of left-$N$-invariant differential forms.
Take $W\subset \bigwedge^1 {\frak n}^{*}$ which is dual to $\frak c$.
Then we have $dW\subset \Omega^{1,1}(\Gamma\backslash N)$.
Thus,  in this case, $\Gamma\backslash N$ admits  a   transverse  K\"ahler structure on the fundamental  central foliation $\mathcal F_{C}$.

We study the properties of nilamnifolds admitting special  transverse K\"ahler structures  on fundamental  central foliations

\begin{prop}
Let $\Gamma\backslash N$ be a nilmanifold with a (not necessarily left-invariant) complex structure $J$.
We assume that $M$ admits a   transverse K\"ahler structure on a $k$-dimensional central foliation ${\mathcal F}_{H}$.
Suppose  that  ${\mathcal F}_{H}$ is regular i.e.\ $H$ is a compact and the $H$-action is free.
Then  $\Gamma\backslash N$ is biholomorphic to a holomorphic principal torus bundle over a complex torus.
In particular,  $\Gamma\backslash N$ is $2$-step nilmanifold (see \cite{PS}).
\end{prop}
\begin{proof}
By the assumption,  $\Gamma\backslash N$ admits a holomorphic principal torus $H$ bundle structure $\Gamma\backslash N\to  B$ so that the base space is a compact K\"ahler manifold.
Since $\Gamma\backslash N$ is an aspherical manifold with $\pi_{1}(\Gamma\backslash N)\cong \Gamma$, $B$ is a compact aspherical manifold such that  $\pi_{1}(B)$ is a finitely generated nilpotent group.
By results in \cite{hasegawa}, \cite{BG} and \cite{BC}, $B$ is a complex torus.
Thus  $\Gamma\backslash N$ is  a  holomorphic principal torus bundle over a complex torus.
\end{proof}

We are interested in the non-regular case.

\begin{prop}
Let $\Gamma\backslash N$ be a nilmanifold with a (not necessarily left-invariant) complex structure $J$.
We assume that $\Gamma\backslash N$ admits a  transverse K\"ahler structure on a fundamental  central foliation ${\mathcal F}_{H}$.
If $H$ is complex $1$-dimensional, then $\Gamma\backslash N$ is diffeomorphic to a $2$-step nilmanifold. 

\end{prop}
\begin{proof}
Let  $M$ be a  compact complex $n$-dimensional manifold which admits a  special  transverse K\"ahler structure on a $k$-dimensional central foliation ${\mathcal  F}_{H}$.
Then we have an isomorphism
\[H^{2n}(M,\C)\cong H^{n-k,n-k}_{B}(M)\otimes \bigwedge^{2k} W_{\C}.
\]
Hence, for the mixed Hodge structure as in Theorem  \ref{theorem-MHS}, $H^{2n}(M,\C)$ is generated by elements of bi-degree $(n+k, n+k)$.

Consider nilmanifold $\Gamma\backslash N$.
Then the DGA $\bigwedge \frak n^{*}$ is the minimal model of $\Omega^{*}(\Gamma\backslash N)$ (see \cite{hasegawa}).
If $\Gamma\backslash N$ admits a special   transverse K\"ahler structure on a  central foliation ${\mathcal F}_{H}$,
then  by Theorem  \ref{theorem-MHS}, the minimal model $\bigwedge \frak n^{*}_{\C}$ of $\Omega^{*}(\Gamma\backslash N)$ admits a bigrading  $\bigwedge \frak n^{\ast}_{\C}=\bigoplus \mathcal M_{p,q}^{\ast}$.
Denote ${\mathcal M}^{\ast}_{w}=\bigoplus _{p+q=w} \mathcal M_{p,q}^{\ast}$ and $m_{\omega}=\dim {\mathcal M}^{1}_{w}$.
Since $\dim \mathcal M^{1}=\dim  \frak n^{\ast}_{\C}=2n$, we have $\sum_{\omega\ge 1}m_{w}=2n$.
Since we have $H^{2n}(M,\C)=\bigwedge ^{2n}\frak n^{\ast}_{\C} =\bigwedge^{2n}\bigoplus_{w}{\mathcal M}^{1}_{w}$,
we have  $\sum_{w\ge 1}wm_{w}=2n+2k$.
Let $k=1$.
Then $\sum_{w\ge 2}(w-1)m_{w}=2$ and hence
we have $m_{2}=2$ and $m_{i}=0$ for $i\le 3$, or  $m_{2}=0$, $m_{3}=1$ and $m_{i}=0$ for $i\le 4$.
We can say $d{\mathcal M}_{1}=0$ and $\bigwedge \frak n^{*}_{\C}=\bigwedge {\mathcal M}_{1}^{1}\otimes \bigwedge  V$ with $dV\subset \bigwedge^{2} {\mathcal M}_{1}^{1}$.
This implies that $\frak n$ is $2$-step.
\end{proof}

We suggest the  following problem.
\begin{pro}
For $s\ge 3$ and $k\ge 2$, does there  exist a $s$-step  nilmanifold admitting  a special  transverse K\"ahler structure on a $k$-dimensional non-regular central foliation ${\mathcal F}_{H}$?
\end{pro}

\subsection{Vaisman manifolds}\label{vais}

Let $(M,J)$ be a compact complex manifold with a Hermitian metric $g$. 
We consider the fundamental form  $\omega=g(-,J-)$ of $g$.
The metric $g$ is  locally conformal K\"ahler (LCK)  if we have a closed $1$-form $\theta$ (called   the Lee form) such that $d\omega=\theta\wedge \omega$.
It is known that if $\theta\not=0$ and $\theta$ is non-exact,
then $(M,J)$ does not admits a K\"ahler structure.
Let $\nabla$ be the Levi-Civita connection of $g$.
A LCK metric $g$ is a Vaisman metric if  $ \nabla \theta=0$.

If $g$ is Vaisman, then the following holds (see \cite{Tsu1}, \cite{Tsu2}): 
\begin{itemize}
\item Let $A$ and $B$ be the dual vector fields of $1$-forms $\theta$ and $-\theta \circ J$ with respect to $g$, respectively. Then $A=JB$, $L_{A}J=0$, $L_{B}J=0$, $L_{A}g=0$, $L_{B}g=0$ and $[A,B]$=0.

\item The holomorphic vector field $B-\sqrt{-1}A$ gives a holomorphic foliation $\mathcal F$.

\item The basic form $d(\theta\circ J)$ is a transverse K\"ahler structure.

\item
We denote by $\operatorname{Aut}_{0}(M, g)$ the identity component of the group of holomorphic isometries, by $\frak h$  the abelian sub-algebra $\langle A, B\rangle$ of the Lie algebra of  $\operatorname{Aut}_{0}(M, g)$ and by $H$ the connected Lie subgroup of $\operatorname{Aut}_{0}(M, g)$  which corresponds to $\frak h$.
Let $T$ be the closure of $H$ in $\operatorname{Aut}_{0}(M, g)$.
Then $T$ is a torus.

\end{itemize}

Thus a compact Vaisman manifold $M$ admits a  transverse K\"ahler structure on the $1$-dimensional fundamental central foliation ${\mathcal F}_{H}$.
Hence, taking $W=\langle \theta,\theta\circ J\rangle$  our results can be applied to a compact Vaisman manifold.
The cohomology of the DGA
\[ A^{*}=H_{B}^{\ast}(M)\otimes \bigwedge\langle  \theta,\theta\circ J\rangle
\]
is isomorphic to the de Rham cohomology of $M$  and
the cohomology of DBA
\[ B^{*,*}=H_{B}^{\ast,\ast}(M)\otimes \bigwedge\langle  \theta+\sqrt{-1}\theta\circ J, \theta-\sqrt{-1}\theta\circ J\rangle
\]
is isomorphic to the Dolbeault cohomology of $M$.
We can easily compute 
\[H^{1}(M,\C)=H^{1}_{B}(M)\oplus \langle \theta\rangle=H^{1,0}_{B}(M)\oplus H^{0,1}_{B}(M) \oplus \langle \theta\rangle.
\]
This implies well known fact that the first Betti number of a compact Vaisman manifold is odd (see \cite{Tsu1}).
We have the mixed Hodge structure
\[H^{1}(M,\C)=H^1_{1,0}\oplus H^1_{0,1}\oplus H^{1}_{1,1}
\]
with $\dim H^{1}_{1,1}=1$ as in Theorem \ref{theorem-MHS}.
We notice that Vaisman metrics are closely related to Sasakian structures.
We can also obtain nice de Rham models of Sasakian manifolds like the above DGA (see \cite{Ti}) and we can develop Morgan's mixed Hodge theory on Sasakian manifolds (see \cite{K}).

Since we have 
$\bar\partial  (\theta+\sqrt{-1}\theta\circ J)=\sqrt{-1}d(\theta\circ J)$ and $\bar\partial  (\theta-\sqrt{-1}\theta\circ J)=0$, we can easily obtain an isomorphism of DGA
\[ A^{*}\otimes \C\cong {\rm Tot}^{*}  B^{*,*}.
\]
Hence, by Theorem \ref{theorem-model}, we have the following (cf.\ \cite[Theorem 3.5]{Tsu1}). 
\begin{cor}
Let $M$ be a compact complex manifold.
We suppose that $M$ admits a Vaisman  metric.
Then the two DGAs $(\Omega^{\ast}(M)\otimes \C,d)$ and $(\Omega^{\ast}(M)\otimes \C,\bar\partial)$ are quasi-isomorphic.
In particular, there exists an isomorphism between the complex valued de Rham cohomology and the Dolbeault cohomology. 
\end{cor}
\begin{rem}
On compact K\"ahler manifold $M$, by the $\partial\bar\partial$-lemma, two DGAs $(\Omega^{\ast}(M)\otimes \C,d)$ and $(\Omega^{\ast}(M)\otimes \C,\bar\partial)$ are quasi-isomorphic (see \cite{NT}).
\end{rem}

\begin{ex}\label{hop}
Let $\Lambda=(\lambda_{1},\dots, \lambda_{n})$ be complex numbers so that $0<\vert \lambda_{n}\vert\le \dots \le \vert \lambda_{1}\vert<1$. A {\em primary Hopf manifold} $M_{\Lambda}$ is the quotient of $\C^{n}-\{0\}$ by the group generated by the transformation $(z_{1},\dots,z_{n})\mapsto (\lambda_{1}z_{1},\dots,\lambda_{n}z_{n})$.
It is known that any $M_{\Lambda}$ admits a Vaisman metric  (see\cite{KO}).
For any $\Lambda$, $M_{\Lambda}$ is diffeomorphic to $S^{1,2n-1}=S^{1}\times S^{2n-1}$.
On the other hand, the complex structure on $M_{\Lambda}$ varies.
If $\lambda_{n}= \dots = \lambda_{1}$, then $M_{\Lambda}$  is a holomorphic principal torus bundle over $\C P^{n-1}$.
Otherwise, any holomorphic principal torus bundle structure over $\C P^{n-1}$ does not exist on $M_{\Lambda}$.
By Example \ref{S1n} and the above arguments, we can obtain explicit representatives of de Rham, Dolbeault, basic de Rham and Basic Dolbeault cohomologies of $M_{\Lambda}$ by using a Vaisman metric on $M_{\Lambda}$.
\end{ex}


\begin{thebibliography}{99}
\bibitem{BC} O. Baues,  V. Cort\'es,
Aspherical K\"ahler manifolds with solvable fundamental group. Geom. Dedicata {\bf 122} (2006), 215--229.
\bibitem{BG}
Ch. Benson, C.~S. Gordon, K\"ahler and symplectic structures on nilmanifolds, {\em Topology} \textbf{27} (1988), no.~4, 513--518.

\bibitem{BM}
S. Bochner and D. Montgomery, Locally compact groups of differentiable transformations, Ann. of Math. (2) {\bf 47} (1946), 639--653. 
\bibitem{CE} E. Calabi, B. Eckmann, A class of compact, complex manifolds which are not algebraic. Ann. Math. {\bf58} (1953), 494--500.
\bibitem{CNMY}
B. Cappelletti-Montano, A. De Nicola, J. C. Marrero, I. Yudin
Almost formality of quasi-Sasakian and Vaisman manifolds with applications to nilmanifolds.  arXiv:1712.09949
\bibitem{Del}
P. Deligne,  Th\'eorie de Hodge. II.  Inst. Hautes \'Etudes Sci. Publ. Math. No. {\bf 40} (1971), 5--57. 
\bibitem{DGMS}P. Deligne, P. Griffiths, J. Morgan, and D. Sullivan, Real homotopy theory of Kahler manifolds, Invent. Math. {\bf 29} (1975), no. 3, 245--274.
\bibitem{EKA}
A. El Kacimi-Alaoui, 
Op\'erateurs transversalement elliptiques sur un feuilletage riemannien et applications. 
Compositio Math. {\bf 73} (1990), no. 1, 57--106. 

\bibitem{Fe}Y. F$\acute {\rm e}$lix, J. Oprea and D. Tanr$\acute {\rm e}$, Algebraic Models in Geometry, Oxford Graduate Texts in
Mathematics 17, Oxford University Press 2008.
\bibitem{GMo} P. Griffiths, J. Morgan,  Rational homotopy theory and differential forms. Second edition. Progress in Mathematics, {\bf 16}. Springer, New York, 2013.

\bibitem{hasegawa}
K. Hasegawa, Minimal models of nilmanifolds, {\em Proc. Amer. Math. Soc.} \textbf{106} (1989), no.~1, 65--71.
\bibitem{Hochschild} G. Hochschild, The structure of Lie groups, Holden-Day Inc., San Francisco, 1965. 
\bibitem{I} H. Ishida, Torus invariant transverse K\"{a}hler foliations, to appear in Trans. Amer. Math. Soc., available at {\tt arXiv:1505.06035}

\bibitem{KO} Y. Kamishima, L Ornea,
Geometric flow on compact locally conformally Kähler manifolds. Tohoku Math. J. (2) {\bf 57} (2005), no. 2, 201--221.
\bibitem{K} H. Kasuya,
Mixed Hodge structures and Sullivan's minimal models of Sasakian manifolds,  Ann. Inst. Fourier (Grenoble), {\bf 67} (2017), no. 6,  2533--2546.
\bibitem{Le}F. Lescure, Exemples d'actions induites non r\'esolubles sur la cohomologie de Dolbeault.

Topology {\bf 35} (1996), no. 3, 561--581. 
\bibitem{LMN}  J. J. Loeb, M. Manjarin, M. Nicolau,
Complex and CR structures on compact Lie groups associated to abelian actions. 
Ann. Global Anal. Geom. {\bf 32} (2007), no. 4, 361--378.
\bibitem{M}  L. Meersseman, A new geometric construction of compact complex manifolds in any dimension. Math. Ann. {\bf 317} (2000), no. 1, 79--115.
\bibitem{MV}  L. Meersseman, A. Verjovsky, 
Holomorphic principal bundles over projective toric varieties. 
J. Reine Angew. Math. {\bf 572} (2004), 57--96. 
\bibitem{Mor}
J. W. Morgan, The algebraic topology of smooth algebraic varieties. \textit{ Inst. Hautes \'{E}tudes Sci. Publ. Math.} No. {\bf 48} (1978), 137--204.
\bibitem{NT} J.  Neisendorfer, L. Taylor,
 Dolbeault homotopy theory. Trans. Amer. Math. Soc.  {\bf 245} (1978), 183--210. 
\bibitem{Or} P. Orlik, 
Seifert manifolds. Lecture Notes in Mathematics, Vol. {\bf 291}. Springer-Verlag, Berlin-New York, 1972.
\bibitem{PS} R. S. Palais, T. E. Stewart,  Torus bundles over a torus. Proc. Amer. Math. Soc. {\bf 12} 1961 26--29. 
\bibitem{Sul} D. Sullivan, Infinitesimal computations in topology. \textit{ Inst. Hautes \'Etudes Sci. Publ. Math.} No. {\bf 47} (1977), 269--331 (1978).
\bibitem{Ta}
D. Tanr\'e,  Mod\'ele de Dolbeault et fibr\'e holomorphe.  J. Pure Appl. Algebra {\bf 91} (1994), no. 1-3, 333–345.
\bibitem{Ti}
A. M. Tievsky, Analogues of K\"ahler	geometry on Sasakian manifolds, Ph.D. Thesis, Massachusetts Institute of Technology,	2008. Available in http://dspace.mit.edu/handle/1721.1/45349
\bibitem{Tsu1} K. Tsukada,
 Holomorphic forms and holomorphic vector fields on compact generalized Hopf manifolds. Compositio Math. {\bf 93} (1994), no. 1, 1--22.

\bibitem{Tsu2}
K. Tsukada,
The canonical foliation of a compact generalized Hopf manifold. 
Differential Geom. Appl. {\bf 11} (1999), no. 1, 13--28. 
\end{thebibliography}
\end{document}